\documentclass[12pt]{article}
\usepackage{color}
\usepackage{amssymb}
\usepackage{graphicx}
\usepackage{subfigure}
\usepackage{lineno,hyperref,amsmath,amsthm}
\usepackage{caption}
\newtheorem{theorem}{Theorem}[section]
\newtheorem{lemma}{Lemma}[section]
\newtheorem{remark}{Remark}[section]
\numberwithin{equation}{section}
\newtheorem{assumption}{\bf Assumption}

\title{Modeling oncolytic virus therapy with distributed delay and non-local diffusion }
\author{Zizi Wang\footnote{Email: zzwang@wzu.edu.cn, gdwangzizi@163.com.}\\
School of Mathematics and Physics, Wenzhou University,\\ Wenzhou, 325035, China}
\date{}

\begin{document}
\maketitle

\begin{abstract}
In the field of modeling the dynamics of oncolytic viruses, researchers often face the challenge of using specialized mathematical terms to explain uncertain biological phenomena. This paper introduces a basic framework for an oncolytic virus dynamics model with a general growth rate $\mathcal{F}$ and a general nonlinear incidence term $\mathcal{G}$. The construction and derivation of the model explain in detail the generation process and practical significance of the distributed time delays and non-local infection terms. The paper provides the existence and uniqueness of solutions to the model, as well as the existence of a global attractor. Furthermore, through two auxiliary linear partial differential equations, the threshold parameters $\sigma_1$ are determined for sustained tumor growth and $\lambda_1$ for successful viral invasion of tumor cells to analyze the global dynamic behavior of the model. Finally, we illustrate and analyze our abstract theoretical results through a specific example.
\end{abstract}
    
\section{Introduction}

Oncolytic virus (OV) therapy is a promising approach for treating cancer, as it selectively targets and destroys cancer cells while preserving healthy cells from harm \cite{MA2023122}. Specifically, upon encountering cancer cells, oncolytic viruses enter these cells and commence replication. As viruses multiply, their numbers increase rapidly within cancer cells, ultimately leading to cellular rupture and death. Moreover, this process can stimulate the immune system to recognize cancer cells, triggering immune cells to attack and eliminate both infected and uninfected cancer cells \cite{RN234}. Nonetheless, as OV therapy is still an evolving cancer treatment, a comprehensive understanding of the underlying biology and pharmacology is essential to analyze the interactions between the growing tumor, replicating virus, and potential immune responses \cite{RN122, RN269}.

To gain a comprehensive understanding of the mechanism behind oncolytic virotherapy, it is imperative to establish a rational mathematical model\cite{RN270}. During the past decade, several valuable mathematical models have been proposed that provide crucial information on oncolytic virotherapy. These typical mathematical models compartmentalize oncolytic virotherapy into uninfected tumor cells ($U$), infected tumor cells ($I$), free virus particles ($V$), virus-specific immune response ($VS$), and tumor-specific immune response ($TS$), among others. Wodarz \cite{RN122} comprehensively considered U-I-TS-VS, constructing four ordinary differential equation models. Through quantitative analysis of the equilibria, insights on the favorable effects of low viral lethality and high replication rates on tumor treatment were presented, as well as relevant immunological mechanisms. Dingli et al. \cite{RN49} and Bajzer et al. \cite{RN86} proposed the U-I-V model for the measles oncolytic virus, and using numerical modeling techniques, investigated the validity of the model and provided interpretations and predictions for experimental data. Furthermore, Wang et al. \cite{RN589}  and Wang et al. \cite{RN27}, respectively, introduced models with time delays for U-I and U-I-V based on the viral replication cycle. Wang et al. \cite{wang2019a}  explored a dynamics model with time delays for U-I-VS and verified that virus-immune suppressive drugs can effectively enhance tumor treatment. In addition, Li and Xiao \cite{RN271} introduced and analyzed the U-I-TS oncolytic virus dynamics model, pointing out the significant enhancement of treatment efficacy by tumor-specific immunity.In addition, Ding et al. \cite{RN273} derived an oncolytic virus dynamic model with nonlocal time delay based on an age-structured model. The numerical model demonstrated the significant role of the time delay term in explaining experimental data.

 The aforementioned models presuppose a homogeneous mixing of cells and viruses within the tumor, disregarding the presence of spatial structure. In reality, tumors often exhibit complex spatial arrangements, which can significantly impact the dynamics of virus diffusion \cite{zhao2019spatial,RN101,A2020Analysis}. Zhao and Tian \cite{zhao2019spatial} studied a delayed reaction-diffusion model for U-I-V, incorporating virus diffusivity, tumor cell diffusion, and the viral lytic cycle. Based on the ODE model proposed by Wang et al. \cite{RN11}, Elaiw et al. \cite{RN101,A2020Analysis} further investigated the influence of spatial heterogeneity on tumor treatment. They developed a tumor immunity reaction-diffusion model \cite{RN101} and a delayed reaction-diffusion system with a virus circulation period \cite{A2020Analysis}.

However, the predictability of oncolytic virus therapy through mathematical modeling is often constrained by specific mathematical terms, which significantly influence the global dynamics of the model \cite{RN270}. To overcome the dependence of the model on specific mathematical terms, Natalia and Wodarz proposed a general framework of ordinary differential equations for oncolytic viruses \cite{RN87}. The model is given as follows:
\begin{eqnarray*}
 \underbrace{ U'(t)}_{uninfected\ tumor\ cells} &=& \underbrace{U \mathcal{F}(U+I)}_{tumor\ proliferation}-\underbrace{\beta I \mathcal{G}(U,V)}_{virus\ infection}, \\
  \underbrace{I'(t)}_{infected\ tumor\ cells } &=& \underbrace{\beta V \mathcal{G}(U,I)}_{virus\ infection}-\underbrace{\alpha I(t)}_{death}.
\end{eqnarray*}

On this basis, Wang et al. \cite{Wang2022AMM} derived a general oncolytic virus therapy model with non-local delay term incorporating an age-structured model. The model is as follows:
\begin{equation}
    \begin{cases}
U'(t)&=U \mathcal{F}\Big( U+\frac{1}{\kappa} \int_{t-\tau}^t \beta V(\theta)\mathcal{G}(U(\theta))e^{-\alpha(t-\theta)}d\theta \Big)-\beta V \mathcal{G}(U),\\
V'(t)&=\beta \mathcal{G}\Big( U(t-\tau) \Big) V(t-\tau)e^{-\alpha\tau}-\alpha V(t).
    \end{cases}\label{2023In1}
\end{equation}
Where $\beta$  is the infection rate, $\alpha$ is the death rate, $\tau$ represents the duration of the viral replication cycle, $\kappa$ is the average number of viruses on every infected tumor cell, $\mathcal{F}$ denotes the tumor growth function, and $\mathcal{G}$ represents the viral infection function.

This paper investigates the global dynamical behavior of the system \eqref{2023In1} under spatially heterogeneous conditions. For $t>0$ and $x\in\Omega$,  where $\Omega$ represents a spatially bounded domain, we consider a system governed by the following equations: 
\begin{equation}
\begin{cases}
  \begin{aligned}
    \partial_t U(t,x) =&  d_1\Delta U(t,x)- \mathcal{G}\Big(x,U(t,x), V(t,x)\Big)\\
&+U(t,x)\mathcal{F}\Big(x,U(t,x)+I(t,x)\Big),\\
\partial_t V(t,x)= & d_2 \Delta V(t,x)-\alpha(x) V(t,x)\\
&+\int_{\Omega}\Gamma(x,y,\tau) \mathcal{G}\Big (y,U(t-\tau,y),V(t-\tau,y)\Big)dy,\\
I(t,x)=& \frac{1}{\kappa}\int_0^{\tau} \int_{\Omega}\Gamma(x,y,a) \mathcal{G}\Big(y,U(t-a,y),V(t-a,y)\Big)dy da,\\
\end{aligned}\label{2020c-eq8}
\end{cases}
\end{equation}
with the following Robin boundary conditions:
\begin{equation}
\frac{\partial U}{\partial \overrightarrow{n}}+\eta_1(x)U =0,\quad
\frac{\partial V}{\partial \overrightarrow{n}}+\eta_2(x)V=0,\quad
\forall x\in\partial\Omega,\ t>0, \label{2020c-maineq2}  
\end{equation}
and the initial conditions are given as 
\begin{equation}
U(\theta,x) = U_0(\theta,x), V(\theta,x) = V_0(\theta,x), \quad \forall x\in\overline{\Omega}, \theta \in [-\tau,0].\label{2020c-maineq3}  
\end{equation}

In this context, $d_1$ and $d_2$ represent the diffusion rates of cells and viruses, respectively. The parameter $\kappa$ signifies the average number of viruses inside infected tumor cells. The function $\alpha(x)$ denotes the virus death rate. All of $d_1$, $d_2$, and $\kappa$ are positive constants. Additionally, $\alpha(x)$ is positive Holder continuous function over the closure of $\Omega$. The function $\Gamma$ refers to the Green function or the fundamental solution of the operator $d_2\Delta - \alpha(x)$ under the appropriate boundary conditions. Moreover, $\frac{\partial}{\partial \overrightarrow{n}}$ denotes the outward normal derivative on $\partial\Omega$, and $\eta_1(x), \eta_2(x) \in \mathbb{C}^{1+\alpha}(\partial \Omega,\mathbb{R}_+)$ with $\eta_1(x),\ \eta_2(x)>0$, and $U_0(\theta,x),\ V_0(\theta,x) \in \mathbb{C}([-\tau,0]\times \Omega,\mathbb{R}_+)$.

Furthermore, the functions $\mathcal{F}$ and $\mathcal{G}$ represent the cell growth rate and the nonlinear incidence term, respectively, satisfying Assumption\ref{2020c-a1}.

\begin{assumption} $\mathcal{F}$ and $\mathcal{G}$ are continuous-differentiable functions and satisfy that
\begin{itemize}
\item[(1)] There exists $K_0>0$ such that $\mathcal{F}\left( x,K_0\right) <0$ and $\mathcal{F}(x,0)>0$ for all $x\in\overline{\Omega}$.
\item[(2)] For any $x\in\overline{\Omega}$, $\mathcal{F}(x,\cdot)$ is decreasing on $\mathbb{R}^+$, and for some $x_0\in\Omega$, $\mathcal{F}(x_0,\cdot)$ is strictly decreasing on $\mathbb{R}^+$;
\item[(3)] $\mathcal{G}(\cdot,0,\cdot)=\mathcal{G}(\cdot,\cdot,0)=0$;
\item[(4)] $\partial_{u_1} \mathcal{G}(x,u_1,u_2)>0$, $\partial_{u_2}  \mathcal{G}(x,u_1,u_2)>0$ for all $x\in \overline{\Omega}$, $u_1,\ u_2\geq0$.
\item[(5)] $\mathcal{G}(x,u_1,u_2) \leq \partial_{u_2} \mathcal{G}(x,u_1,0) u_2$ for all $x\in\overline{\Omega},\ u_1,\ u_2>0$;
\item[(6)] $\frac{\mathcal{G}(x,u_1,u_2)}{u_1}$ is bounded, and $\partial_{u_1}\frac{\mathcal{G}(x,u_1,u_2)}{u_1}\geq0$ for all $x\in\overline{\Omega},\ u_1,u_2>0$.
\end{itemize}\label{2020c-a1}
\end{assumption}

\begin{remark}
Assumption \ref{2020c-a1} possesses universal significance, as it is satisfied by many common growth rate functions (see Table \ref{2023a.Table1}) and nonlinear infection functions (see Table \ref{2023a.Table2}).
\begin{table}[htbp]
\caption{Some examples of $\mathcal{F}(\cdot,U)$}
\centering
\begin{tabular}{cll}
\hline
$F(\cdot,U)$   & Type           &  {REF.}   \\\hline
$b-d U$     &Logistic growth &  \cite{RN25,RN87} \\
$b-d\ln(U)$     & Gomperzian Growth  &  \cite{dingli2009dynamics,Kuang2016}  \\
  \hline
\end{tabular}
\label{2023a.Table1}
\\
\footnotesize{Where $b,\ d$ are non-negative constants.}
\end{table}

\begin{table}[htbp]
\caption{Some examples of $\mathcal{G}(\cdot,U,V)$}
\centering
\begin{tabular}{cll}
\hline
 $\mathcal{G}(\cdot,U,V)$   & Type           &  {REF.}   \\\hline
$\beta U V$  & Mass action   & \cite{A2020Analysis,wang2019a} \\
$\frac{\beta UV}{1+h V}$     & Holling type II &  \cite{RN108,RN177} \\\hline
\end{tabular}
\\
\footnotesize{Where $\beta, \ h$ is non-negative constant.}
\label{2023a.Table2}
\end{table}
\end{remark}

\begin{remark}
Model \eqref{2020c-eq8}-\eqref{2020c-maineq3} is very versatile as it encompasses various biological scenarios. For example, it can represent a predator-prey model \cite{gourley2004stage, RN111} or an infectious disease model \cite{RN232} by appropriately selecting the parameters $\mathcal{F}$, $\mathcal{G}$, and $\kappa$. Specifically, our model is equivalent to the one proposed in \cite{gourley2004stage} when spatial factors $x$ are ignored and $\kappa$ approaches infinity. In this case, $d_1$ , $d_2$ are both zero,  $\alpha(x)$ is constant and $\Gamma(x,y,\tau)$ is given by $e^{-d\tau}$. Furthermore, if $\kappa$ is set to infinity and $\mathcal{F}(x,U)=\frac{\mu(x)}{U(x)}-d(x)$, $\mathcal{G}(x,U,V)=\beta(x)U(x)V(x)$, and $\eta_1(x)=\eta_2(x)=0$, the model will transform into an infectious disease model \cite{RN232}. Lastly, when $\kappa$ is set to infinity, the model corresponds to the predator-prey model proposed in \cite{RN111}.
\end{remark}

The remainder of this paper is organized as follows. Section \ref{2020cModel} provides a detailed derivation of the model \eqref{2020c-eq8}. In Section \ref{2023section3}, we establish the existence and uniqueness of solutions to the model, as well as the global compact attractor of the solutions. In Section \ref{2023section4}, we present the harsh conditions for global stability of the zero steady state and tumor steady state, as well as the conditions for successful viral invasion of tumor cells and the existence of positive steady state. In Section \ref{2023section5}, we consider specific examples of the model, explain its practical applications, and provide lower bound estimates for tumor cell and viral particle populations after viral invasion. Finally,  in Section \ref{2023section6}, we conclude the paper by summarizing the overall framework and findings.

\section{Model development}\label{2020cModel}

Inspired by Komarova and Wodarz \cite{RN87}, the uninfected tumor cells satisfy:
\begin{equation}
\begin{aligned}
    \underbrace{\partial_t U(t,x)}_{Uninfected\ tumor\ cells}&=\underbrace{d_1\Delta U(t,x)}_{Diffusion}+\underbrace{U(t,x)\mathcal{F} \Big( x,U(t,x)+I(t,x)\Big)}_{Tumor\ proliferation}\\
    &-\underbrace{\mathcal{G}\Big(x,U(t,x),V(t,x)\Big)}_{Virues\ therapy}
\end{aligned},\label{2020c-eq1}
\end{equation}

To depict the intracellular viral life cycle \cite{RN27,RN365}, we introduce the notion of infection age denoted by the variable $a$. Let $P(t,a,x)$ be the density (with respect to infection age $a$) of virues at location $x$ and time $t$. We assume that $P(t,a,x)$ complies with the standard argument on age-structured model with spatial diffusion.
\begin{eqnarray}
&&\partial_t P(t,a,x)+\partial_a P(t,a,x)= d_2\Delta P(t,a,x)-\alpha(x) P(t,a,x),\label{2020c-eq2}\\
&&\frac{\partial V}{\partial \overrightarrow{n}}+\eta_2(x)V=0,\label{2020c-eq4}\\
&&P(t,0,x) = \mathcal{G}\Big(x,U(t,x),V(t,x)\Big).\label{2020c-eq3}
\end{eqnarray}

Using the method of characteristic curve, we solve the equations \eqref{2020c-eq2}-\eqref{2020c-eq3}. For any constant $\xi$, letting $t=a+\xi$, we define $W_{\xi}(a,x):=P(\xi+a,a,x)=P(t,a,x)$. Thus,
\begin{eqnarray*}
\partial_a W_{\xi}(a,x)&=&d_2 \Delta W_{\xi}(a,x)-\alpha(x) W_{\xi}(a,x),\\
W_{\xi}(0,x)&=&\mathcal{G}\Big(x,U(\xi,x), V(\xi,x)\Big).
\end{eqnarray*}
Regarding $\xi$ as a parameter, we solve the above equation and obtain: 
\begin{equation*}
W_{\xi}(a,x)=\int_{\Omega}\Gamma(x,y,a)\mathcal{G}\Big(y,U(\xi,y),V(\xi,y)\Big)dy.
\end{equation*}
Thus,
\begin{equation*}
P(t,a,x)=\int_{\Omega}\Gamma(x,y,a) \mathcal{G}\Big(y,U(t-a,y),V(t-a,y)\Big)dy,
\end{equation*}
where $\Gamma$ is the fundamental solution of the operator $d_2 \Delta-\alpha(x)$ associated with boundary condition $\frac{\partial W_{\xi}}{\partial \overrightarrow{n}}+\eta_2(x)W_{\xi}=0$ where $\eta_2(x)\in C(\overline{\Omega},R^+)$ is Holder continuous and $\frac{\partial W}{\partial \overrightarrow{n}}$ denotes the derivative along the outward normal direction $\overrightarrow{n}$ to $\partial \Omega$ \cite{RN2906}.

We divide the virus into two parts: the virus within tumor cells $E(t,x)$ and free viruses $V(t,x)$.  And $E(t)$ and $V(t)$ satisfy the following equations:
\begin{eqnarray*}
E(t,x)&=&\int_{0}^{\tau}P(t,a,x)da\\
&=&\int_0^{\tau} \int_{\Omega}\Gamma(x,y,a) \mathcal{G}\Big(y,U(t-a,y),V(t-a,y)\Big)dyda,\\
V(t,x)&=&\int_{\tau}^{+\infty}P(t,a,x)da\\
&=&\int_{\tau}^{+\infty} \int_{\Omega}\Gamma(x,y,a) \mathcal{G}\Big(y,U(t-a,y),V(t-a,y)\Big)dyda,
\end{eqnarray*}
where $\tau$ represents the virus replication cycle, in other words, for the infection age $a<\tau$, viruses are all within the infected tumor cells; and after the infection age $a>\tau$, the viruses lyse the infected tumor cells and become free virus particles.

We postulate that there are $\kappa$ virus particles within each infected tumor cell, leading to the relationship:
\begin{equation}
I(t,x)=\frac{E(t,x)}{\kappa}.\label{2020c-eq5}
\end{equation}
Besides, one can rewrite $V(t,x)$ as
\begin{equation}
V(t,x)=\int_{-\infty}^{t-\tau}\bigg(\int_{\Omega}\Gamma (x,y,t-\tau) \mathcal{G}\Big(y,U(s,y),V(s,y)\Big)dy\bigg) ds.\label{2020c-eq6}
\end{equation}
Differentiating \eqref{2020c-eq6} with respect to $t$ and using equation \eqref{2020c-eq2}, we have
\begin{equation}
\partial_t V(t,x)= d_2(x)\Delta V(t,x)-\alpha(x) V(t,x)+\int_{\Omega}\Gamma(x,y,\tau)  \mathcal{G}\Big(y,U(t-\tau,y),V(t-\tau,y)\Big)dy.\label{2020c-eq7}
\end{equation}

Thus, by equation \eqref{2020c-eq1}, equation \eqref{2020c-eq5}, and equation \eqref{2020c-eq7} together with the boundary condition $\frac{\partial}{\partial \overrightarrow{n}}\cdot+\eta_2(x)\cdot=0$, we obtain the full oncolytic virus therapy system \eqref{2020c-eq8}.

\section{Well posed and global compact attractor}\label{2023section3}

Let $\Omega$ be a bounded domain in $\mathbb{R}^n$ and $\mathbb{X}=\mathbb{C}(\overline{\Omega},\mathbb{R}^2_+)$ be the Banach space of continuous functions with values in the real plane, equipped with the norm $\|u\|_X$, which is the supremum norm. Assume that $\Omega$ has a smooth boundary $\partial\Omega$, and let $Y=\mathbb{C}(\overline{\Omega},\mathbb{R})$. Set $\tau\geq0$ and $\mathbb{C}_{\tau}=\mathbb{C}([-\tau,0],\mathbb{X})$ with the norm $\|\phi\| :=\max_{\theta\in[-\tau,0]}\|\phi(\theta)\|_{\mathbb{X}}$. We define $u_t\in \mathbb{C}_{\tau}$ by
$$u_t(\theta)=u(t+\theta),\quad \forall \theta\in[-\tau,0].$$

Let $T_1(t),\ T_2(t):Y\rightarrow Y,\ t\geq 0$, be the semigroups corresponding to $A_1:=d_1\Delta,\ A_2:=d_2\Delta-\alpha(x)$ with the boundary condition $\frac{\partial }{\partial \overrightarrow{n}}+\eta_1(x)$ and $\frac{\partial}{\partial \overrightarrow{n}}+\eta_2(x)$, respectively. In other words, the linear operator $A:=(A_1,A_2)$ with domain $D(A)=D(A_1)\times D(A_2)$ is the infinitesimal generator of  $C_0$ semigroups $T(t):=(T_1(t),T_2(t))$.

Define $\mathcal{L}:=(\mathcal{L}_1,\mathcal{L}_2)$ as a mapping from $\mathbb{C}_{\tau}$ to $\mathbb{C}_{\tau}$ with
\begin{equation*}
    \begin{aligned}
        \mathcal{L}_1(\phi_1,\phi_2)(\cdot)&=-\mathcal{G}\Big (\cdot,\phi_1(0,\cdot),\phi_2(0,\cdot)\Big)
        +\phi_1(0,\cdot) \mathcal{F}\bigg(\cdot,\phi_1(0,\cdot)\\
        &+\frac{1}{\kappa} \int_0^{\tau}\int_{\Omega}\Gamma(\cdot,y,a) \mathcal{G}\Big(y,\phi_1(-a,y),\phi_2(-a,y)\Big)dyda\bigg)\\
        \mathcal{L}_2(\phi_1,\phi_2)(\cdot)&=-\alpha(\cdot)\phi_2(0,\cdot)+\int_{\Omega}\Gamma(\cdot,y,\tau) \mathcal{G}\Big(y,\phi_1(-\tau,\cdot),\phi_2(-\tau,\cdot)\Big)dy 
    \end{aligned}
\end{equation*}
for $\phi=(\phi_1,\phi_2)\in \mathbb{C}_{\tau}$. Assuming that the semi-flow of the solution generated by the system \eqref{2020c-eq8}-\eqref{2020c-maineq3} is $\Phi(t)$, we can recast the system \eqref{2020c-eq8}-\eqref{2020c-maineq3} as a semi-dynamical system $u(t):=\Phi(t) \phi $ where $\phi(\cdot)=(\phi_1(\cdot),\phi_2(\cdot))\in \mathbb{C}_{\tau}$.  And $u(t)$ satisfies the following evolution equation:
\begin{equation}\label{2020c-eq9}
\begin{cases}
\frac{du(t)}{dt}=Au+\mathcal{L}(u_t),\qquad t>0,\\
u(0)= \phi.
\end{cases}
\end{equation}
\begin{theorem}For any function $\phi\in\mathbb{C}_{\tau}$, the system \eqref{2020c-eq8}-\eqref{2020c-maineq3} admits a unique non-continuable solution $u$ defined on $\bar{\Omega} \times[-\tau, t_{\infty})$, where $t_{\infty}=$ $t_{\infty}(\phi)$ with $0<t_\infty \leq \infty$. Furthermore, for $(t, x) \in[-\tau, t_\infty) \times \bar{\Omega}$, $u(t, x, \phi)$ belongs to $\mathbb{R}_{+}^2$.\label{2023T1}
\end{theorem}
\begin{proof}
    By the Corollary 4 in Martin and Smith \cite{RN213}, it suffices to prove that the operator $\mathcal{L}$ satisfies the subtangential condition and the Lipschitz condition in the invariant set $K:= \mathbb{R}^2_+$. 
    
    It can be easily demonstrated that the operator $\mathcal{L}$ is Lipschitz continuous, because the functions $\mathcal{F}$ and $\mathcal{G}$ are continuously differentiable, as well as the bounded properties of the Green function $\Gamma$.
    
    Now let's verify the  subtangential condition.  For every  $\phi \in \mathbb{C}_{\tau}, x \in \Omega$  and  $h \geq 0$, then
\begin{equation}
    \phi(0, \cdot)+h \mathcal{L}(\phi)(\cdot)=\left(\begin{array}{l}
\phi_1(0, x)+h \mathcal{L}_1\left(\phi_1, \phi_2\right) \\
\phi_2(0, x)+h \mathcal{L}_2\left(\phi_1, \phi_2\right)
\end{array}\right)\label{2023eqad1}
\end{equation}
where
\begin{equation*}
    \begin{aligned}
        \phi_1(0, \cdot)+&h \mathcal{L}_1\left(\phi_1, \phi_2\right)=\phi_1(\cdot, 0) \left[1- h \mathcal{G}\Big(\cdot,\phi_1(0,\cdot),\phi_2(0,\cdot)\Big)\phi_1^{-1}(0,\cdot)\right]\\
        &+h \phi_1(\cdot, 0) \mathcal{F} \Big(\cdot, \phi_1(0, \cdot)+\frac{1}{k} \int_0^\tau \int_{\Omega} \Gamma(\cdot, y, \tau) \mathcal{G}(y,\phi_1(-\tau,\cdot),\phi_2(-\tau,\cdot))dyda\Big),\\
        \phi_2(0,\cdot)+&h\mathcal{L}_2\left(\phi_1,\phi_2\right) = \phi_2(0,\cdot)\left[1-\alpha(\cdot) h \right]
        +h  \int_{\Omega}\Gamma(\cdot,y,\tau) \mathcal{G}\Big(y,\phi_1(-\tau,\cdot),\phi_2(-\tau,\cdot)\Big)dy 
    \end{aligned}
\end{equation*}
According to Assumption \ref{2020c-a1} (6), we can observe that the expression \eqref{2023eqad1} is greater than 0 when $h$ is sufficiently small.   
    \end{proof}
To proceed further, we consider the following equation:
\begin{equation}\label{2020c-eq10}
\begin{cases}
\partial_t z(x,t)= d_1\Delta z(x,t)+ z(x,t)\mathcal{F}\Big(x,z\left(x,t\right)\Big),& x\in \Omega\\
\frac{\partial z(t,x)}{\partial \overrightarrow{n}}+\eta_1(x) z(t,x)=0, &\forall x\in\partial\Omega.
\end{cases}
\end{equation}
It is easy to check that $z(x)=0$ is a steady-state of system \eqref{2020c-eq10}. Linearizing the above system at $0$, we get the following elliptic eigenvalue problem:
\begin{equation}\label{2020c-eq11}
\begin{cases}
\sigma \psi=d_1\Delta \psi+\psi\mathcal{F}(x,0), & x\in\Omega,\\
\frac{\partial \psi(x)}{\partial \overrightarrow{n}}+\eta_1 \psi(x)=0, &\forall x\in\partial\Omega.
\end{cases}
\end{equation}
we define $\sigma_1$ as the principal eigenvalue of the above linear equation \eqref{2020c-eq11} and $\xi_1(x)$ is the corresponding strictly positive eigenfunction.

\begin{lemma}(Proposition 3.3 in \cite{FZ1997JDE} or Proposition 3.3.1 and Proposition 3.3.2 in \cite{RN4})
	System \eqref{2020c-eq10} has a unique positive equilibrium $z_*(x)\leq K_0$ where $K_0$ is defined by Assumption \ref{2020c-a1}(1). And for every initial data $\phi \in \mathbb{C}_{\tau}$, we have $\lim _{t \rightarrow \infty} z(x, t, \phi)=z^*(x)$ uniformly for $x \in \bar{\Omega}$ if $\sigma_1>0$; Whereas, all non-negative solutions of \eqref{2020c-eq10} decay exponentially to zero as $\sigma_1\leq0$.\label{2020c-L3.1}
\end{lemma}

\begin{theorem}\label{2020c-T1}
For any $\phi\in \mathbb{C}_{\tau}$, system \eqref{2020c-eq8} admits a unique classical solution $u(t,x,\phi)$ on  $(t,x)\in [0,+\infty)\times \overline{\Omega}$. Furthermore, the semi-flow solution $\Phi(t): \mathbb{C}_{\tau}\rightarrow\mathbb{C}_\tau$ has a compact global attractor $\mathcal{A}$.
\end{theorem}
\begin{proof}By the Theorem \ref{2023T1}, one gets that system \eqref{2020c-eq8}-\eqref{2020c-maineq3} admits a unique classical solution in $[0,t_{\infty})$ with $U,\ V \geq0$. Then, by the first equation of system \eqref{2020c-eq8},
	$$\partial_t U(t,x)\leq d_1\Delta U(t,x)+U(t,x)\mathcal{F}\Big(x,U(t,x)\Big).$$
By Lemma \ref{2020c-L3.1} and compare theorem, there is a positive constant $B$ such that $U(t,x)\leq B$  for any initial function $\phi\in\mathbb{C}$ as $t>t(\phi)$.

Making use of the boundedness of $U(t,x)$ and Assumption \ref{2020c-a1}(5) in the second equation of \eqref{2020c-eq8}, one gets
\begin{equation}\label{2020c-eq12}
\partial_t V(t,x)\leq d_2\Delta V(t,x)-\alpha(x) V(t,x)+ c\overline{V}(t-\tau),
\end{equation}
for some positive constant $c$, where $\overline{V}(t)=\int_{\Omega} V(t,x)dx$. To show the boundness of $\overline{V}(t)$, integrating the first equation of \eqref{2020c-eq8},
\begin{align*}
\partial_t \overline{U}(t)& =\int_{\Omega} d_1\Delta U(t,x)dx- \int_{\Omega}\mathcal{G}\Big(x,U(t,x), V(t,x)\Big)dx\\
\qquad&+\int_{\Omega} U(t,x)\mathcal{F}\Big(x,U(t,x)+I(t,x)\Big)dx\\
&\leq -\int_{\Omega}\mathcal{G}\Big(x,U(t,x), V(t,x)\Big)dx+\int_{\Omega} U(t,x)\mathcal{F}\Big(x,U(t,x)+I(t,x)\Big)dx
\end{align*}
where  $\overline{U}(t)=\int_{\Omega} U(t,x)dx$, and the inequality is based on the divergence theorem and boundary condition. Thus,
\begin{equation}\label{2020c-eq13}
\int_{\Omega}\mathcal{G}\Big(x,U(t,x), V(t,x)\Big)dx\leq \int_{\Omega} U(t,x)\mathcal{F}\Big(x,U(t,x)+I(t,x)\Big)dx-\partial_t \overline{U}(t).
\end{equation}
Similarly, integrating the second equation of \eqref{2020c-eq8} with respect to $x\in\Omega$, and using \eqref{2020c-eq13}, then there exist two positive numbers $k_1,\ k_2$ such that
\begin{align*}
\partial_t \overline{V}(t)& \leq-\int_{\Omega}\alpha(x)V(t,x)dx+\int_{\Omega}\int_{\Omega}\Gamma(x,y,\tau)\mathcal{G}\Big(y,U(t-\tau,y),V(t-\tau,y)\Big)dydx\\
&\leq -\alpha_0 \overline{V}(t)+k_1\int_{\Omega}\mathcal{G}\Big(y,U(t-\tau,y),V(t-\tau,y)\Big)dy\\
&\leq -\alpha_0 \overline{V}(t)+k_1\int_{\Omega} U(t-\tau,y)\mathcal{F}\Big(x,U(t-\tau,y)+I(t-\tau,y)\Big)  dy-k_1\partial_t \overline{U}(t-\tau)\\
&\leq -\alpha_0 \overline{V}(t)+k_2-k_1\partial_t \overline{U}(t-\tau).
\end{align*}
where 
\begin{equation*}
    \begin{aligned}
        &\alpha_0 = \max_{x\in\overline{\Omega}} \alpha(x),\\
        &k_1 = \max_{(x,y) \in \Omega \times \Omega} \Gamma( x, y, \tau) mes(\Omega),\\
        &k_2 = k_1 mes(\Omega) \max_{ (x,y)\in\overline{\Omega}\times [0,B]}y \mathcal{F}(x,y),
    \end{aligned}
\end{equation*}
 and $mes(\Omega)$ denotes the Lebesgue measure of $\Omega$.

Hence,
\begin{equation*}
    \frac{d}{dt}(\overline{V}(t)e^{\alpha_0 t})\leq k_2e^{\alpha_0 t}-k_1 e^{\alpha_0 t}\frac{d\overline{U}(t-\tau)}{dt}
\end{equation*}
Integrating by parts the above inequality over $t\in [t_1(\phi),t]$, we can find a positive numbers $k_3$, independent of $\phi$, and a positive number $k_4=k_4(\phi)$ dependent on $\phi$ such that
$$\overline{V}(t)\leq k_4(\phi) e^{-\alpha_0 t}+k_3,\qquad t\geq t(\phi).$$
It confirms the boundedness of $\overline{V}(t)$. Combining this with \eqref{2020c-eq12}, there exists a positive number $B_1$, independent of $\phi$, such that $V(t,x)\leq B_1$ if $t\geq t_2(\phi)$. And the solution semi-flows $\Phi(t)$is point dissipaative. Moreover, by Theorem 2.1.8 in \cite{wu1996book} and Theorem 3.4.8 in \cite{hale2010asymptotic}, $\Phi(t) $ has a compact global attractor $\mathcal{A}$.
\end{proof}

\section{Global extinction and persistence}\label{2023section4}
It is evident that the system \eqref{2020c-eq8} has a steady state at $E_0=(0,0)$. Additionally, Lemma \ref{2020c-L3.1} can be utilized to infer that if $\sigma_1>0$, there exists another steady state $E_1=(U_1(x),0)$, where $U_1(x)$ is equivalent to $z_*(x)$ as defined in Lemma \ref{2020c-L3.1}. We now proceed to examine the global asymptotic behavior of system \eqref{2020c-eq11}.
\begin{theorem} If $\sigma_1<0$, then $E_0$ is globally attractive for the model \eqref{2020c-eq8}-\eqref{2020c-maineq3}. In the other words, for any $\phi=(\phi_1,\phi_2)\in \mathbb{C}_{\tau}$, the solution $u(t,\phi)$ of model \eqref{2020c-eq8} satisfies
$$\lim_{t\rightarrow+\infty} u(t,\phi)(x)=(0,0)$$
uniformly for $x\in\overline{\Omega}$.\label{2020c-T3.2}
\end{theorem}
\begin{proof} By the first equation of system \eqref{2020c-eq8}, it is obvious that
\begin{equation}
\partial_t U(t,x)\leq d_1\Delta U(t,x)+U(t,x)\mathcal{F}\Big(x,U(t,x)\Big).\label{2020c-eq17}
\end{equation}

By Lemma \ref{2020c-L3.1} and the standard comparison theorem, one gets that $\lim_{t\rightarrow+\infty} U(t,x)=0$ uniformly for $x\in\overline{\Omega}$ if $\sigma_1<0$.

Now, we regard the $V(t,x)$ as a solution of the following nonautonmous reaction diffusion equation:
\begin{equation}\label{2020c-eq14}
\begin{cases}
\partial_t V(t,x)=d_2 \Delta V(t,x)-\alpha(x) V(t,x)\\
\qquad +\int_{\Omega}\Gamma(x,y,\tau)\mathcal{G}\Big(y,U(t-\tau),V(t-\tau,y)\Big)dy,& x\in\Omega, t>0\\
\frac{\partial V(t,x)}{\partial \overrightarrow{n}}+\eta_2(x)V(t,x)=0.& x\in\partial\Omega,t>0.
\end{cases}
\end{equation}
Since it is known that $\lim_{t\rightarrow +\infty}U(t,x)=0$ uniformly for $x\in\overline{\Omega}$, and $V(t,x)$ remains bounded, it can be inferred that the  system \eqref{2020c-eq14} is asymptotic to the following linear autonomous reaction diffusion equation
\begin{equation}\label{2020c-eq15}
\begin{cases}
\partial_t \hat{V}(t,x)= d_2\Delta \hat{V}(t,x)-\alpha(x) \hat{V}(t,x) & x\in\Omega, t>0\\
\frac{\partial \hat{V}(t,x)}{\partial \overrightarrow{n}}+\eta_2(x)\hat{V}(t,x)=0. & x\in\partial\Omega,t>0.
\end{cases}
\end{equation}
By Theorem 2.2.1 in \cite{RN4}, all non-negative solution of system \eqref{2020c-eq15} will decay to $0$. By a generalized Marku's theorem for asymptotically autonomous semi-flows (Theorem 4.1 in \cite{Thieme1993jmb}), one obtains that $\lim_{t\rightarrow+\infty} V(t,x)=0$.
\end{proof}

To understand the asymptotic stability on steady state $E_1$, we consider the non-local elliptic eigenvalue problem with forcing function $U_1(x)$:
\begin{equation}\label{2020c-eq16}
\begin{cases}
\lambda\psi=d_2 \Delta \psi-\alpha(x)\psi+\int_{\Omega}\Gamma(x,y,\tau) \partial_V \mathcal{G}\Big(y,U_1(y),0\Big) \psi(y)dy e^{-\lambda\tau}, & x\in\Omega\\
\frac{\partial \psi}{\overrightarrow{n}}+\eta_2(x)\psi=0, & x\in\partial\Omega.
\end{cases}
\end{equation}
\begin{lemma}(Theorem 2.2 of \cite{RN111}) Model \eqref{2020c-eq16} exists a principle eigenvalue $\lambda_1(U_1(x))$ associated with a strictly positive eigenfunction $\psi_1$.\label{2020c-l3.2}
\end{lemma}

\begin{theorem}If $\sigma_1>0$ and $\lambda_1(U_1(x))<0$, then $E_1$ is global attractive in $\mathbb{C}_{\tau}$.\label{2020c-T3.3}
\end{theorem}
\begin{proof}By the first equation of system \eqref{2020c-eq8}, $U(t,x)$ satisfies inequality \eqref{2020c-eq17}. Thus, for any $\epsilon>0$ there exists a $t_1=t(\phi)$ depending on initial data $\phi$ such that $U(t,x)\leq U_1(x)+\epsilon$ for $t>t_1$. Taking this result into the second equation of system \eqref{2020c-eq8}, one gets
\begin{eqnarray*}
&&\partial_t V(t,x)\\
&\leq& d_2 \Delta V(t,x)-\alpha(x)V(t,x)+\int_{\Omega}\Gamma(x,y,\tau)\mathcal{G}\Big(y,U_1(x)+\epsilon,V(t-\tau,y)\Big)dy\\
&\leq& d_2 \Delta V(t,x)-\alpha(x)V(t,x)+\int_{\omega}\Gamma(x,y,\tau)\partial_V\mathcal{G}\Big(y_1,U_1(x)+\epsilon,0\Big)V(t-\tau,y)dy.
\end{eqnarray*}
The first inequality is due to Assumption \ref{2020c-a1} (4) and the second inequality comes from Assumption \ref{2020c-a1} (5).
On the other hand, the model
\begin{equation*}
    \partial_t\overline{V}(t,x)=d_2 \Delta V(t,x)-\alpha(x) \overline{V}(t,x)+\int_{\Omega}\Gamma(x,y,\tau)\partial_V \mathcal{G}(y,U_1(x)+\epsilon,0)V(t-\tau,y)dy,
\end{equation*}
exists a principle eigenvalue $\lambda_1(U_1(x)+\epsilon)$ associated with a strictly positive eigenvector $\phi_2(x)$ by Lemma \ref{2020c-l3.2}. And $\lambda_1(U_1(x))<0$ indicates that $\lambda_1(U_1(x)+\epsilon)<0$ for small $\epsilon$. Thus, by standard comparison theorem, $0\leq \lim_{t\rightarrow+\infty}V(t,x)\leq\lim_{t\rightarrow+\infty} \overline{V}(t,x)\leq \lim_{t\rightarrow+\infty}c e^{\lambda_1(U_1(x)+\epsilon)t}\phi_2(x)=0$ where $c$ is a positive constant. Regarding $V(t,x)$ as a fixed function with $\lim_{t\rightarrow+\infty}V(t,x)=0$ for all $x\in\Omega$, one gets that
\begin{equation*}
\begin{cases}
\partial_t \hat{U}(t,x)= d_1(x)\Delta \hat{U}(t,x)+\hat{U}(t,x)\mathcal{F}(x,\hat{U}(t,x)),& x\in\Omega, t>0\\
\frac{\partial \hat{U}(t,x)}{\partial \overrightarrow{n}}+\eta_1(x) \hat{U}(t,x)=0, & x\in\partial\Omega, t>0
\end{cases}
\end{equation*}
is the asymptotically autonomous system of \eqref{2020c-eq8}. Since $\sigma_1>0$, then $\lim_{t\rightarrow+\infty} \hat{U}(t,x)=U_1(x)$ for all $x\in\Omega$. Moreover, Lemma 3.1 in \cite{RUAN1999} ensures that $U(t,x)$ will not tends to $0$ as $t\rightarrow+\infty$. Thus, by Theorem 4.1 in \cite{Thieme1993jmb}, we obtain $\lim_{t\rightarrow+\infty}U(t,x)=U_1(x)$ for all $x\in \Omega$.
\end{proof}

\begin{theorem}Assume $\lambda_1(U_1(x))>0,\ \sigma_1>0$, then the semi-flow $\Phi(t)$ is uniformly persistent (i.e. There exists $\epsilon > 0$ such that $\liminf_{t \rightarrow +\infty} U(t,x) \geq \epsilon$ and $\liminf_{t \rightarrow +\infty} V(t,x) \geq \epsilon$). Moreover, system \eqref{2020c-eq8}-\eqref{2020c-maineq3} admits at least virus therapy equilibrium solution $E_2=(U_2(x),V_2(x))$ such that $U_2(x)>0,\ V_2(x)>0$ and $U_2(x)<U_1(x)$  for all $x\in\Omega$.
\label{2020c-T3.4}
\end{theorem}
\begin{proof}Let $Y_0:=\{(\phi_1,\phi_2)\in C_{\tau}:\phi_i\not\equiv0, i=1,2\}$, $\partial Y_0:=C_{\tau} \setminus Y_0$ and $S:=\{\phi\in\partial Y_0:\Phi(t)\phi\in\partial Y_0, t\geq0\}$. Lemma \ref{2020c-L3.1} indicates that $\omega(S)=\{E_0,E_1\}$ and $\omega(S)$ is acyclic. To obtain the persistence of $\Phi$, by Theorem 4.4.3 in \cite{RN4}, it is sufficient to that $E_0,\ E_1$ are isolated invariant subsets for $\Phi$ in $Y_0$ and that
$$W^S(E_0)\cap Y_0=\emptyset,$$
$$W^S(E_1)\cap Y_0=\emptyset,$$
where $W^S(E_i)$ denotes the stable manifold of equilibrium $E_i$ for $i=0,1$. Suppose
$W^S(E_0)\cap Y_0 \neq \emptyset$. Then for any $\epsilon>0$, there is positive number $t_1=t(\phi)$ depending on initial data $\phi$ such that $t>t_1$ implies that $U(t,x),\ V(t,x)\leq \epsilon$.  Choose $0<\epsilon_1\leq \sigma_1$. Then there exists $t_0>0$ and $\delta_0>0$ such that for all $t>t_0$, $x\in\overline{\Omega}$, and $U\leq \delta_0$, $V \leq \delta_0$, the following inequality holds:
$$
U(t,x) \mathcal{F}\Big(x,U(t,x)+I(t,x)\Big)-\mathcal{G}\Big(x,U(t,x),V(t,x)\Big)\geq U(t,x)[\mathcal{F}(x,0)-\epsilon_1].
$$
By substituting this inequality into the first equation of system \eqref{2020c-eq8}, we obtain:
$$\partial_t U(t,x)\geq d_1 \Delta U(t,x)+U(t,x)[\mathcal{F}(x,0)-\epsilon_1]$$
Considering the following system:
\begin{equation*}
\begin{cases}
\underline{U}(t,x)= d_1 \Delta\underline{U}(t,x)]+\underline{U}(t,x)[\mathcal{F}(x,0)-\epsilon_1],& t\geq t_0,x\in\Omega,\\
\frac{\partial \underline{U}(t,x)}{\partial\overrightarrow{n}}+\eta_1(x)\underline{U}(t,x)=0,  & t\geq t_0,x\in\partial\Omega.    
\end{cases}
\end{equation*}

Then, by comparing theorem, one gets that 
$$U(t,x)\geq \underline{U}(t,x)= c e^{(\sigma_1-\epsilon_1) t}\xi_1 (x)$$ which indicates that $U(t,x)$ will not converges to 0. It contracts with the assumption $W^S(E_0)\cap Y_0 \neq \emptyset$. Furthermore, using the similar method, we also can prove $W^S(E_1)\cap Y_0 = \emptyset$ thanks to $\lambda_1(U_1(x))>0$.  Using acyclicity test for permanence \cite{RN4}, the semi-flow $\Phi(t)$ is persistent.

Furthermore,  Theorem 4.4.6 in \cite{RN4} (or Theorem 6.3 in \cite{smith2011dynamical}) ensures that system \eqref{2020c-eq8} admits at least one equilibrium solution $(U_2(x),V_2(x))$ such that $U_2(x)>0,\ V_2(x)>0$ for all $x\in\Omega$.
Returning to the original equations \eqref{2020c-eq8}-\eqref{2020c-maineq3}, $U_2(x)$ is a positive equilibrium solution of the following equation:
\begin{equation*}
\begin{cases}
    \partial_t \hat{U}_2(t,x)= d_1\Delta \hat{U}_2(x)+\hat{U}_2(x)\mathcal{F}_2\Big(x,\hat{U}_2(x)\Big),& x\in\Omega\\
\frac{\partial \hat{U}_2(x)}{\partial \overrightarrow{n}}+\eta_1(x)\hat{U}_2(x)=0,& x\in\partial\Omega,
\end{cases}
\end{equation*}
where $\mathcal{F}_2(x,u)=\mathcal{F}\Big(x,u+I_2(x,u)\Big)-\frac{1}{u}\mathcal{G}\Big(x,u,V_2(x)\Big)$, and
$$I_2(x)=\frac{1}{\kappa}\int_0^{\tau}\int_{\Omega}\Gamma(x,y,a)\mathcal{G}\Big(y,u(y),V_2(y)\Big)dyda>0.$$
In addition, $U_1(x)$ is a positive equilibrium solution of the system \eqref{2020c-eq10}. By using Assumption \ref{2020c-a1} (1), (2), and (6), it can be concluded that $\mathcal{F}$ and $\mathcal{F}_2$ are strictly decreasing functions of $u$ for $u\geq0$, and $\mathcal{F}(x,u)>\mathcal{F}_2(x,u)$ for $u>0$. Therefore, by applying Proposition 3.3.3 in \cite{RN4}, we have $U_1(x)>U_2(x)$.

\end{proof}

\section{A simple example of the correlation between logistics growth and mass action infection.}\label{2023section5}

In this section, we investigate the specific form of model \eqref{2020c-eq8}, assuming $\alpha(x)=\alpha$, $\mathcal{F}(x,U)=b-dU$, $\mathcal{G}(x,U,V)= \frac{\beta U V}{1+h V}$, $\eta_1(x)=\eta_2(x)=0$, and $\Omega=(0,\pi)$,  where  $\alpha, \beta, d, h $ are positive constants, and $b$ is a non-negative constant.

Then, the model equations \eqref{2020c-eq8} tend to the following forms:
\begin{equation}
   \begin{cases}
 \partial_t U(t, x)&=d_1 \Delta U(t, x)-\frac{\beta U(t, x) V(t, x)}{1+h V(t, x)}+U(t, x) \Big [b-d U(t, x) \\
 &- \frac{\beta}{\kappa} \int_0^\tau \int_0^\pi \Gamma(x, y, a)\frac{ U(t-a, y) V(t-a, y)}{1+h V(t-a, y)} d y d a\Big ], \\
 \partial_t V(t, x)&= d_2 \Delta V(t, x)-\alpha  V(t, x)+\beta \int_0^\pi \Gamma(x, y, \tau ) \frac{U(t-\tau, y) V(t-\tau, y)}{1+h V(t-\tau, y)} d y.
\end{cases} \label{2023eq4.1}
\end{equation}
with following homogeneous Neumann boundary conditions:
\begin{equation}
    \partial_x U(t, 0)  =\partial_x U(t, \pi)=\partial_x V(t, 0)=\partial_x V(t, \pi)=0, \qquad t>0,\label{2023eq4.2}
\end{equation}
and initial conditions are given by:
\begin{equation}
    U(0,x) = U_0(x),V(0,x)=V_0(x), \qquad x\in[0,\pi].   \label{2023eq4.3}
\end{equation}
The function $\Gamma(x,y,a)$ in model \eqref{2023eq4.1} is the Green's function of the following equation:
\begin{equation}
   \begin{cases}
\partial_a W(a, x)=d_2 \Delta W- \alpha W, & x\in(0,\pi), t>0, \\
W(0, x)  =W_0(x)& x\in(0,\pi), \\
\partial_x W(t, 0)  =\partial_x W(t, \pi)=0, &  t>0.\label{2023eq4.4}
\end{cases} 
\end{equation}

By using the method of separation of variables, we know that the solution of the equation \eqref{2023eq4.4} is:
\begin{equation}
\begin{aligned}
& W(a, x)=\frac{2}{\pi} \sum_{n=1}^{\infty} \Big [\int_0^\pi W_0(y) \cos n y d y \Big ] \cos (n x) e^{-\left(n^2 d_2+\alpha\right) a} \\
& =\frac{2}{\pi} \int_0^\pi \sum_{n=1}^{\infty} \Big [e^{-\left(n^2 d_2+\alpha\right) a} \cos n y \cos n x \Big ] W_0(y) d y \\
& =\int_0^\pi \Gamma(x, y, a) W_0(y) d y \\
\end{aligned}\label{2023eq4.5}
\end{equation}
where $\Gamma(x, y, a)=\frac{2}{\pi} \sum_{n=1}^{\infty} e^{-\left(n^2 d_2+\alpha\right)a} \cos n y \cos n x$ is the Green function of system \eqref{2023eq4.4}.

In addition, in the special case when $W_0(y)=1$, we know that the solution of equation \eqref{2023eq4.4} is $W(a,x)= e^{-\alpha a}$. Therefore, by utilizing the uniqueness of the solution and equation \eqref{2023eq4.5}, we can deduce that $\int_0^\pi \Gamma(x, y, a) d y=e^{-\alpha a}$. By performing a simple computation, we can thus obtain the following conclusion.

$\frac{2}{\pi} \int_0^\pi \sum_{n=1}^{\infty} \Big [e^{-\left(n^2 d_2+\alpha\right) a} \cos n y \Big ]d y $

 \begin{theorem}
The constant equilibria of the system \eqref{2023eq4.1} - \eqref{2023eq4.3} can be described as follows:
\begin{enumerate}
    \item When $b \leq 0$, the model has a unique steady state  $E_0=(0,0)$;\label{2023T511}
    \item When $b > 0$ and $\frac{b \beta}{\alpha d e^{\alpha \tau}}\leq 1$, the system \eqref{2023eq4.1} - \eqref{2023eq4.3} has a original steady state $E_0=(0,0)$ and an untreated steady state $E_1=\left(\frac{b}{d}, 0\right)$;\label{2023T512}
    \item When $\frac{b \beta}{\alpha d e^{\alpha \tau}}>1$, besides the original steady state $E_0=(0,0)$ and the untreated steady state $E_1=\left(\frac{b}{d}, 0\right)$, the system \eqref{2023eq4.1} - \eqref{2023eq4.3} also has a unique positive treated steady state $E_3=\left(U_3, V_3\right)$, where $U_3$ and $V_3$ are both positive.\label{2023T513}
\end{enumerate}
\end{theorem}

\begin{proof}
To study the constant steady state, we consider the following system:
$$
\left\{\begin{array}{l}
-\frac{\beta U V}{1+h V}+ U\left[b-d U-\frac{\beta}{\kappa}\left(1-e^{-\alpha \tau}\right) \frac{U V}{1+h V}\right]=0, \\
-\alpha V+\beta e^{-\alpha \tau} \frac{U V}{1+h V}=0.
\end{array}\right.
$$

When $V=0$, it is easy to obtain conclusion \ref{2023T511} and conclusion \ref{2023T512}. When $V \neq 0$, we have $U=\frac{\alpha e^{\alpha \tau}}{\beta}(1+h V)$, where $V$ satisfies the quadratic equation.
\begin{equation}
    A V^2+B V+C=0.\label{2023eq4.6}
\end{equation}
where
\begin{equation*}
\begin{aligned}
    &A=-d{\alpha}^{2} \left( {{\rm e}^{\alpha\,\tau}} \right) ^{2}{h}^{2}-{
\frac {\beta\,{\alpha}^{2} {{\rm e}^{\alpha\,\tau}}  h}{\kappa}}(e^{\alpha \tau}-1)
,\\
&B=\alpha e^{\alpha \tau}\left[-\beta^2+\beta bh-2 d \alpha e^{\alpha \tau} h-\frac{\alpha \beta}{\kappa}\left(e^{\alpha \tau}-1\right)\right],\\
&C= \alpha (\beta b-d \alpha e^{\alpha \tau}).
\end{aligned}
\end{equation*}
It is easy to check that $A<0$. Hence, in the case where $C>0$ (i.e., when $\frac{b \beta}{\alpha d e^{\alpha \tau}}>1$), equation \eqref{2023eq4.6} possesses a single positive solution, which can be obtained using the Vieta's formula. Please refer to Figure \ref{fig:square} for a visual representation of this scenario.
\begin{figure}
  \centering
  \includegraphics[width=0.5\textwidth]{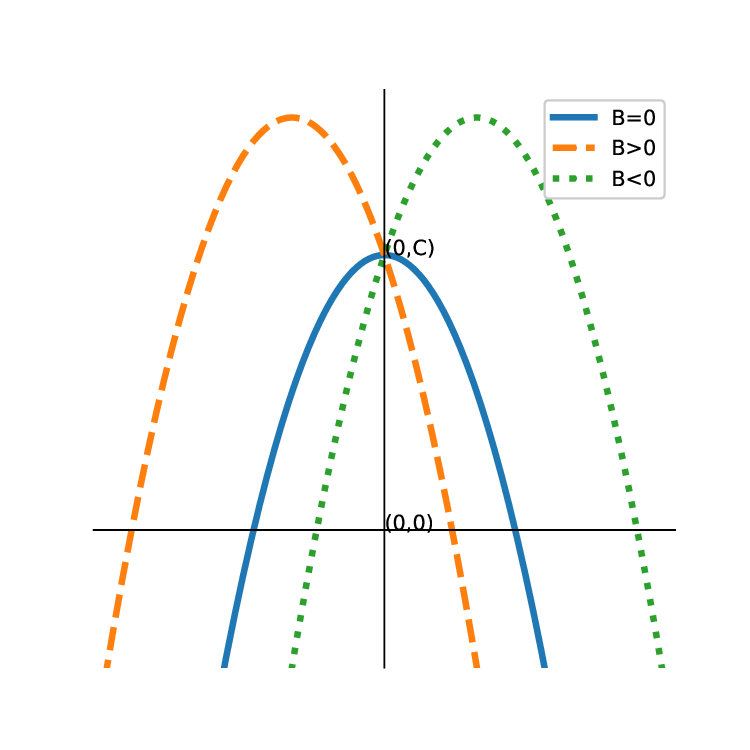}
  \caption{This figure shows the quadratic curves for all fixed $A<0$ and $C<0$. We use dashed lines to represent the case of $B<0$, dotted lines to represent the case of $B>0$, and solid lines to represent the quadratic curves with $B=0$.}
  \label{fig:square}
\end{figure}

Conversely, when $\frac{b \beta}{\alpha d e^{\alpha \tau}}\leq1$, it implies that $B< 0$ and $C\leq 0$. Applying the Vieta's formula once more, it can be readily deduced that the equation \eqref{2023eq4.6} lacks a positive solution.

\end{proof}

In addition, when $V(t,x)=0$, it follows that $U(t,x)$ satisfies:
\begin{equation}
\begin{cases}
\partial_t U(t, x)=d_1 U(t, x)+U(t, x)\Big [b-d U(t, x)\Big ], & x \in(0,1), t>0, \\ \partial_x U(t, 0)=\partial_x U(t, \pi)=0 . & t>0.
\end{cases}\label{2023eq4.7}
\end{equation}
Linearizing the above system \eqref{2023eq4.7} at 0, we obtain the elliptic eigenvalue problem:
\begin{equation*}
\left\{\begin{array}{l}
\sigma \psi=d_1 \Delta \psi+b \psi, \quad x \in(0, \pi) \\
\partial_x \psi(0)=\partial_x \psi(\pi)=0
\end{array}\right.
\end{equation*}
It is easy to check that the principle eigenvalue $\sigma_1 = b$ and the corresponding eigenfunction $\psi=1$. By Lemma \ref{2020c-L3.1}, we have the following Lemma.

\begin{lemma}
If $b>0$, then the system \eqref{2023eq4.7} possesses a unique positive steady state $U_1(x)=\frac{b}{d}$, such that for every $\phi\in \mathbb{C}_{\tau}$, $\lim_{t\rightarrow +\infty} U(x, t, \phi)=\frac{b}{d}$ uniformly for $x\in [0,\pi]$. Whereas, $\lim_{t\rightarrow +\infty} U(x, t, \phi)=0$ uniformly for $x\in [0,\pi]$ as $b<0$.\label{2023L5.1}
\end{lemma}

Now, we consider the following non-local elliptic eigenvalue problem:
\begin{equation}
\left\{\begin{array}{l}
s \psi=d_2 \Delta \psi-\alpha \psi+e^{-s \tau}\frac{b \beta}{d } \int_{0}^{\pi} \Gamma(x, y, \tau) \psi(y) d y, \quad x \in(0, \pi) \\
\partial_x \psi(0)=\partial_x \psi(\pi)=0 .
\end{array}\right.\label{2023eq4.8} 
\end{equation}

\begin{lemma}
    The principle eigenvalue of system \eqref{2023eq4.8} has the same sign as $\lambda_1=-\alpha +\frac{\beta b}{d e^{\alpha\tau}}$.\label{2023L5.2}
\end{lemma}
\begin{proof}
According to Theorem 2.2 in \cite{RN111} (or Lemma 2.4 in \cite{RN232}), system \eqref{2023eq4.8} possesses a principal eigenvalue $s_1$ along with a strictly positive eigenfunction $\psi_1$. Moreover, the principal eigenvalue $s_1$ shares the same sign as ${\lambda_1}$, where ${\lambda_1}$ represents the principal eigenvalue of the following systems:
\begin{equation}
\left\{\begin{array}{l}
\lambda \psi=d_2 \Delta \psi-\alpha \psi+\frac{b \beta}{d } \int_{0}^{\pi} \Gamma(x, y, \tau) \psi(y) d y, \quad x \in(0, \pi) \\
\partial_x \psi(0)=\partial_x \psi(\pi)=0 .
\end{array}\right.\label{2023eq4.9} 
\end{equation}
To estimate ${\lambda}_1$, we integrate the first equation of system \eqref{2023eq4.9}, then 
\begin{equation}
  \begin{aligned}
{\lambda} \overline{\psi} &= -\alpha \overline{\psi} +\frac{b\beta}{d}\int_0^{\pi}\int_0^{\pi}\Gamma(x,y,\tau)\psi(y)dydx\\
&=-\alpha \overline{\psi}+\frac{b\beta}{d}\int_0^{\pi}\Big [\int_0^{\pi}\Gamma(x,y,\tau)dx \Big ]\psi(y)dy\\
&= -\alpha \overline{\psi}+\frac{b\beta}{de^{\alpha\tau}}\overline{\psi}
  \end{aligned}\label{2023eq4.10}
\end{equation}
where $\overline{\psi} := \int_0^{\pi} \psi(x)dx$. The second equal sign in equality \eqref{2023eq4.10} arises from exchanging the integrals, and the third equal sign in equality \eqref{2023eq4.10} stems from the symmetry of the Green's function and $\int_0^{\pi}\Gamma(x,y,a)dy=e^{-\alpha a}$. 

Due to the fact that the principal characteristic vector $\psi_1$ of system \eqref{2023eq4.9} is strictly positive, it follows that $\overline{\psi_1}\neq 0$. Hence, equation \eqref{2023eq4.10} implies that the principal eigenvalue $\lambda_1$ of system \eqref{2023eq4.9} is given by $\lambda_1 = -\alpha +\frac{\beta b}{d e^{\alpha\tau}}$.
\end{proof}

\begin{theorem} The global dynamic behavior of the model \eqref{2023eq4.1}-\eqref{2023eq4.3} is described as follows:
\begin{enumerate}
    \item If $b<0$, original steady state $E_0$ is globally stable;\label{conclusion1}
    \item If $b>0$ and $\frac{\beta b}{\alpha d e^{\alpha\tau}}<1$, then tumor steady state $E_1$ is global attractive in $\mathbb{C}_{\tau}$;\label{conclusion2}
    \item If $b>0$ and $\frac{\beta b}{\alpha d e^{\alpha\tau}}>1$, the solution semigroup $\Phi(t)$ generated by system \eqref{2023eq4.1}-\eqref{2023eq4.3} is uniformly persistent. Furthermore, when $h>\max\{\frac{\beta b-\alpha d e^{\alpha\tau}}{b^2},\frac{\beta \kappa +e^{\alpha\tau}-1}{b\kappa} \}$, we have
    \begin{equation*}
        \begin{aligned}
            \lim_{t\rightarrow\infty} U(t,x)&\geq \frac{\alpha \kappa (b^2 h -\beta b+\alpha d e^{\alpha\tau})}{\alpha \kappa b d h+(1-e^{-\alpha\tau})(\beta b -\alpha d e^{\alpha\tau})},\\
            \lim_{t\rightarrow\infty}V(t,x)&\geq \frac{(\beta b-\alpha d e^{\alpha\tau})(hb\kappa-\beta \kappa -e^{-\alpha \tau}+1)}{h e^{\alpha\tau}\Big [\alpha \kappa b d h+(1-e^{-\alpha\tau})(\beta b-\alpha d e^{\alpha\tau})\Big ]}.
        \end{aligned}
    \end{equation*}\label{conclusion3}
\end{enumerate}\label{2023T5.2}
\end{theorem}
\begin{proof}
Conclusion \ref{conclusion1} arises from Theorem \ref{2020c-T3.2} and Lemma \ref{2023L5.1}; Conclusion \ref{conclusion2} stems from Theorem \ref{2020c-T3.3}, Lemma \ref{2023L5.1}, and Lemma \ref{2023L5.2}.By combining Theorem \ref{2020c-T3.4}, Lemma \ref{2023L5.1}, and Lemma \ref{2023L5.2}, we can derive the first part of Conclusion \ref{conclusion3}.

Now, we will prove the second part of Conclusion \ref{conclusion3}. Based on Theorem  \ref{2020c-T1}, it is known that the solution semigroup $\Phi(t)$ generated by system \eqref{2023eq4.1}-\eqref{2023eq4.3} is bounded. Therefore, we can choose $c>0$ such that $c U - \beta U V + U [b - d U-\frac{\beta}{\alpha \kappa}(1-e^{-\alpha \tau}) U V ]$ is monotone increasing in $U$ for all values taken by the solution. Utilizing the Green's function $\Gamma_1(x,y,a)$ associate with $d_1 \Delta$ and Neumann boundary condition, we have 
\begin{equation}
\begin{aligned}
U(t, x)= & e^{-c t} \int_{\Omega} \Gamma_1(x, y,t) U_0(x) d y+\int_0^t e^{-c s} \int_{\Omega} \Gamma_1(x, y,s)\bigg \{c U(t-s, y) \\
- & \frac{\beta U(t-s, y) V(t-s, y)}{1+h V(t-s,y)}+U(t-s, y) \Big [b-d U(t-s, y)\\
-& \frac{\beta}{\kappa} \int_0^{\tau} \int_0^\pi \Gamma(y, z, a) \frac{U(t-s-a, z) V\left(t-s-a, z\right)}{1+h V(t-s-a,z)} d z d a \Big ] \bigg\} d y d s\\
\end{aligned}\label{2023eq4.11}
\end{equation}

Let 
\begin{equation*}
    \begin{aligned}
        U^{\infty}(x) := \limsup_{t\rightarrow+\infty} U(t,x), &\qquad  U_{\infty}(x) := \liminf_{t\rightarrow+\infty} U(t,x),\\
        V^{\infty}(x) := \limsup_{t\rightarrow+\infty} V(t,x), & \qquad V_{\infty}(x) := \limsup_{t\rightarrow+\infty} V(t,x).
    \end{aligned}
\end{equation*}

As the solution  semigroup $\Phi(t)$ is uniformly persistent, for all $x\in\overline{\Omega}$, there exists a $\epsilon>0$ such that
\begin{equation*}
       U^{\infty}(x) \geq U_{\infty}(x)\geq \epsilon, \qquad  V^{\infty}(x)\geq V_{\infty}(x) \geq \epsilon.
\end{equation*}

By applying Fatou's Lemma to equation \eqref{2023eq4.11}, we obtain
\begin{equation}
\begin{aligned}
U^{\infty} (x)\leq & \int_0^{\infty} e^{-c s} \int_{\Omega} \Gamma_1(x, y,s)\bigg \{c U^{\infty}(y) 
- \frac{\beta U^{\infty}( y) V_{\infty}( y)}{1+h V_{\infty}(y)}\\
+ & U^{\infty}(y) \Big [b-d U^{\infty}(y)-\frac{\beta}{\kappa} \int_0^{\tau} \int_0^{\pi} \Gamma(y, z, a)  \frac{U^{\infty}(z) V_{\infty}(z)}{1+h V_{\infty}(z)} d z d a \Big ] \bigg\} d y d s\\
\end{aligned}\label{2023eq4.12}
\end{equation}

Let
\begin{equation*}
\begin{aligned}
\mathbb{U}^{\infty} := \sup_{x\in [0,\pi]} U^{\infty}(x),&\qquad \mathbb{V}^{\infty} := \sup_{x\in [0,\pi]} V^{\infty}(x),\\
\mathbb{U}_{\infty} := \inf_{x\in [0,\pi]} U_{\infty}(x),&\qquad \mathbb{V}_{\infty} := \inf_{x\in [0,\pi]} V_{\infty}(x).
\end{aligned}
\end{equation*}

Besides, due to the validity of  
$$\int_0^\pi \Gamma_1 (x, y, a) d y=1,\qquad \int_0^\pi \Gamma(x, y, a) d y=e^{-\alpha a},$$
the inequality \eqref{2023eq4.12} can be rewritten as 
\begin{equation}
\mathbb{U}^{\infty} \leq \frac{1}{c}\{c \mathbb{U}^{\infty}-\frac{\beta \mathbb{U}^{\infty} \mathbb{V}_{\infty}}{1+h \mathbb{V}_{\infty}}+\mathbb{U}^{\infty}\left[b-d \mathbb{U}^{\infty}-\frac{\beta}{\alpha \kappa}\left(1-e^{-\alpha \tau}\right) \frac{\mathbb{U}^{\infty} \mathbb{V}_{\infty}}{1+ h \mathbb{V}_{\infty}} \big ]\right\}.\label{2023eq4.13}
\end{equation}

The simplified form of \eqref{2023eq4.13} yields:
\begin{equation}
    \frac{\beta\mathbb{V}_{\infty}}{1+h \mathbb{V}_{\infty}}\leq 
    [1+(\frac{1-e^{-\alpha \tau }}{\alpha \kappa})\mathbb{U}^{\infty}]^{-1}  (b-d \mathbb{U}^{\infty}).\label{2023eq4.14}
\end{equation}

Similarly, by the Fatou's Lemma of the lower limit, we can obtain
\begin{equation}
\frac{\beta\mathbb{V}^{\infty}}{1+h \mathbb{V}^{\infty}}\geq [1+(\frac{1-e^{-\alpha \tau }}{\alpha \kappa})\mathbb{U}_{\infty}]^{-1} (b-d \mathbb{U}_{\infty}).\label{2023eq4.15}
\end{equation}

By employing the Green's function $\Gamma$ for $d_2\Delta -\alpha$, we obtain
\begin{equation*}
    V(t, x)=\int_{0}^\pi \Gamma(x, y, t) V_0(y) d y+\beta \int_0^t \int_0^\pi \Gamma(x, y, s+\tau) \frac{U(t-s-\tau, y) V(t-s-\tau, y) }{1 + h V(t-s-\tau,y)}d y d s.
\end{equation*}
Again, by Fatou's Lemma together with the fundamental properties of the Green's function $\Gamma$, one can obtain the following results:
\begin{equation}
       1 \leq \frac{\beta \mathbb{U}^\infty}{\alpha e^{\alpha \tau} (1+h \mathbb{V}^\infty)},\label{2023eq4.16} 
\end{equation}
and 
\begin{equation}
    1 \geq \frac{\beta \mathbb{U}_\infty}{\alpha e^{\alpha\tau} (1+h \mathbb{V}_\infty)}.\label{2023eq4.17} 
\end{equation}

Besides, by comparing theorem together with the first equation of system \eqref{2023eq4.1}, it is easy to obtain that
$\lim_{t\rightarrow\infty} u(t,x) \leq \frac{b}{d}$, which indicates that 
\begin{equation}
\mathbb{U}^{\infty} \leq \frac{b}{d}.\label{2023eq4.18}
\end{equation}

Taking inequality \eqref{2023eq4.18} into \eqref{2023eq4.16}, we can obtain 
\begin{equation}
    \mathbb{V}^\infty \leq \frac{1}{h}(\frac{\beta b}{\alpha d e^{\alpha\tau}}-1).\label{2023eq4.19}
\end{equation}
Again, Taking \eqref{2023eq4.19} into \eqref{2023eq4.15}, and yield
\begin{equation}
    \mathbb{U}_\infty \geq \frac{\alpha \kappa (b^2 h-\beta b+\alpha d e^{\alpha\tau})}{\alpha \kappa b d h+(1-e^{-\alpha\tau})(\beta b-\alpha d e^{\alpha\tau})}.\label{2023eq4.20}
\end{equation}
Lastly, taking \eqref{2023eq4.20} into \eqref{2023eq4.17}, and yield 
\begin{equation*}
       \mathbb{V}_\infty \geq \frac{(\beta b-\alpha d e^{\alpha\tau})(hb\kappa-\beta \kappa -e^{\alpha \tau}+1)}{h e^{\alpha\tau} \Big [\alpha \kappa b d h+(1-e^{-\alpha\tau})(\beta b-\alpha d e^{\alpha\tau})\Big ]}.   
\end{equation*}

By the definition of $\mathbb{U}_{\infty}$ and $\mathbb{V}_{\infty}$, We have completed the proof for this part.  
\end{proof}

\section{Conclusion}\label{2023section6}
In this paper, we derive a therapeutic model for oncolytic virus treatment based on an age-structured model with non-local time delays and non-local infection spreading. It is worth mentioning that to overcome the limitation of mathematical models in explaining uncertain biological phenomena using specific functional expressions, we use general continuous differentiable functions $\mathcal{F}$ and $\mathcal{G}$ to characterize tumor growth and virus infection.

Theorem \ref{2020c-T1} ensures the existence and uniqueness of the model's solution, as well as the existence of a global compact attractor. Additionally, the principle eigenvalue $\sigma_1$ defined by system \eqref{2020c-eq11} determines the sustained proliferation of tumor cells:

(1) Theorem \ref{2020c-T3.2} indicates that tumor growth is not sustainable when $\sigma_1 < 0$.
(2) Lemma \ref{2020c-L3.1} shows that tumor growth reaches a saturation state when $\sigma_1 > 0$.

Furthermore, in the case of sustained tumor growth (i.e., $\sigma_1 > 0$), the principle eigenvalue $\lambda_1$ defined by system \eqref{2020c-eq16} determines the success of oncolytic virus treatment:
(1) Theorem \ref{2020c-T3.3} reveals that the viral treatment fails when $\lambda_1 < 0$ (i.e., $\lim_{t\rightarrow+\infty}V(t,x) = 0$).
(2) Theorem \ref{2020c-T3.4} demonstrates that the viral treatment is successful when $\lambda_1 > 0$ (i.e., there exists $\epsilon>0$ such that $\lim_{t\rightarrow+\infty}V(t,x) >\epsilon$).

Then, we assume that the tumor follows logistic growth (with growth rate $\mathcal{F}=b-dU$) and Holling type II functional response (viral infection function $\mathcal{G}(U,V)=\frac{\beta U V}{1+h V}$) under Neumann boundary conditions. Model \eqref{2020c-eq8}-\eqref{2020c-maineq3} is transformed into model \eqref{2023eq4.1}-\eqref{2023eq4.3}. We calculate the tumor threshold parameter as $\sigma_1 = b$ and the viral treatment threshold parameter as $\lambda_1 = -\alpha+\frac{\beta b}{d e^{\alpha\tau}}$. Furthermore, we provide a lower bound estimate for the solution under tumor treatment conditions.

We believe that our model is highly versatile as it incorporates different tumor growth processes and viral infection processes. The dynamic results of this study provide theoretical foundations for further data fitting and prediction of the model.

    \bibliography{ref}
    \bibliographystyle{plain}
\end{document}